\documentclass[10pt]{article}
\usepackage[T1]{fontenc}
\usepackage[utf8]{inputenc}
\usepackage[english]{babel}
\usepackage{authblk}
\usepackage{amsmath,amssymb,mathrsfs,amsthm}
\usepackage{titlesec}
\usepackage[filecolor=green]{hyperref}
\usepackage[toc,page]{appendix}
\usepackage{color}
\usepackage[a4paper]{geometry}
\usepackage{tikz, tikz-cd}
\usepackage{enumitem}
\setlist[itemize]{itemsep=.2cm, leftmargin=.75cm}
\setlist[description]{itemsep=.5cm, leftmargin=.5cm, font=\normalfont\underline}
\usepackage{stmaryrd}

\addtocounter{section}{1}
\addtocounter{subsection}{1}

\newtheorem{theorem}{Theorem}[subsection]
\newtheorem{lemma}[theorem]{Lemma}
\newtheorem{proposition}[theorem]{Proposition}
\newtheorem{corollary}[theorem]{Corollary}

\theoremstyle{definition}
\newtheorem{definition}[theorem]{Definition}
\newtheorem{remark}[theorem]{Remark}
\newtheorem{example}[theorem]{Example}

\newcommand\eg{{\em e.g.}}
\newcommand\ie{{\em i.e.}}
\newcommand\andquad{\qquad\mathrm{and}\qquad}
\newcommand\com{,\ }

\newcommand\N{\mathbb{N}}
\newcommand\Ninfty{\mathbb{N}\cup\{\infty\}}
\newcommand\K{\mathbb{K}}
\newcommand\R{\mathbb{R}}
\newcommand\X{x_1,{\cdots},x_n}
\newcommand\KX{\K[\X]}
\newcommand\KXX{\K[[\X]]}
\newcommand\KXnc{\K\coeff{\X}}
\newcommand\KXXnc{\K\coeff{\coeff{\X}}}
\newcommand\Kxy{\K\coeff{\coeff{x,y}}}
\newcommand\topFPS{\tau_\delta}
\newcommand\olI{\overline{I}}
\newcommand\coeff[1]{\langle#1\rangle}
\newcommand\integers[2]{\{#1,{\cdots},#2\}}
\DeclareMathOperator\supp{supp}
\DeclareMathOperator\val{val}

\tikzcdset{%
	tikzcd circle/.tip={Circle[open,length=6.75pt,sep=-3.75pt] cm to},
	circle/.code={\pgfsetarrowsend{tikzcd circle}},
}
\newcommand\reflRew{\overset{=}{\to}}
\newcommand\transRew{\overset{\star}{\to}}
\newcommand\topRew{\begin{tikzcd}[ampersand replacement=\&,cramped,sep=small]\arrow[r, circle]\&{}\end{tikzcd}}
\newcommand\opord{<_\mathrm{op}}
\newcommand\dTop[1]{\tau^{#1}_\text{dis}}
\DeclareMathOperator\lc{lc}
\DeclareMathOperator\lm{lm}
\DeclareMathOperator\rem{r}

\begin{document}

\title{Topological closure of formal powers series ideals\\
and application to topological rewriting theory}
\author{Cyrille Chenavier}
\author{Thomas Cluzeau}
\author{Adya Musson-Leymarie \thanks{Corresponding author: \texttt{adya.musson-leymarie@unilim.fr}}}
\affil{Univ. Limoges, CNRS, XLIM, UMR 7252, F-87000 Limoges}
\date{}

\maketitle
      
\begin{abstract}
  We investigate formal power series ideals and their relationship to
  topological rewriting theory. Since commutative formal power series algebras 
  are Zariski rings, their ideals are closed for the adic topology defined by
  the maximal ideal generated by the indeterminates. We provide a constructive
  proof of this result which, given a formal power series in the topological
  closure of an ideal, consists in computing a cofactor representation of the
  series with respect to a standard basis of the ideal. We apply this result in
  the context of topological rewriting theory, where two natural notions of 
  confluence arise: topological confluence and infinitary confluence. We give
  explicit examples illustrating that in general, infinitary confluence is a
  strictly stronger notion than topological confluence. Using topological
  closure of ideals, we finally show that in the context of rewriting theory on
  commutative formal power series, infinitary and topological confluences are
  equivalent when the monomial order considered is compatible with the degree. 
\end{abstract}
\noindent
\begin{small}\textbf{Keywords:} Complete rings, formal power series, topological
  rewriting.\\[0.2cm]
  \textbf{M.S.C 2020 - Primary:} 13F25, 13J10. \textbf{Secondary:} 68Q42. 
\end{small}

\tableofcontents

\newpage

\begin{center}
  \bf\large 1. INTRODUCTION
  \addcontentsline{toc}{section}{1. Introduction}
\end{center}
\bigskip

Algebraic rewriting systems are computational models used to prove algebraic
properties through rewriting reasoning. For that, one considers a presentation
by generators and relations of an algebraic system, \eg, a monoid or an algebra,
and associates to each relation a rewriting rule, consisting in simplifying one
term of the relation into the other terms of the relation. When the underlying
algebraic system is equipped with the discrete topology, two fundamental
rewriting properties are termination, \ie, there is no infinite rewriting
sequence, and confluence, \ie, whenever two finite rewriting sequences diverge
from a common term, then these sequences can be extended by finitely many
rewriting steps to reach a common term. When these two properties hold together,
rewriting theory provides effective methods for computing linear bases, Hilbert
series, homotopy bases or free resolutions~\cite{Anick86, Guiraud19, Guiraud12,
  Kobayashi90}, and for obtaining constructive proofs of coherence theorems,
from which we deduce an explicit description of the action of a monoid on a 
category~\cite{Gaussent15}, or of homological properties such as finite
derivation type, finite homological type~\cite{Guiraud18, Squier87}, or
Koszulness~\cite{Priddy70}. 
\smallskip

Topological rewriting theory is an extension of discrete rewriting, where the
underlying set of terms admits a non-discrete topology. Such topological
rewriting systems appear in computer science, in the context of rewriting over
infinitary $\Sigma/\lambda$-terms~\cite{Dershowitz91, Kennaway97}, and in
abstract algebra, in the context of rewriting over commutative formal power 
series~\cite{Chenavier20}. With this topological framework, it is natural to
consider not only finite rewriting sequences, but also rewriting sequences that
converge for the underlying topology, which brings us to two different notions
of confluence: topological and infinitary confluences. Each of these notions
allows us to extend diverging rewriting sequences by infinite rewriting
sequences having a common term as a limit. However, the rewriting sequences
diverging from the same term have different natures: they are assumed to be
finite for topological confluence, and they converge to limits when one
deals with infinitary confluence. Infinitary confluence is a strictly stronger 
property in general, explicit examples are given in Subsection 4.1 of the
present paper.
\smallskip

Our objective is to prove that in the context of rewriting over commutative 
formal power series, topological and infinitary confluence are actually
equivalent properties. 

\paragraph{Topological closure of commutative formal power series ideals.}

Proving that infinitary confluence and topological confluence are equivalent in
the context of commutative formal power series requires to establish that an
ideal of a formal power series algebra $\KXX$, where $\K$ is a field of
arbitrary characteristic, is closed for the~$(\X)$-adic
topology. The latter, denoted by $\topFPS$, is induced by the metric $\delta$
which is defined by
\begin{equation}\label{equ:delta}
    \delta(f,g)=\frac{1}{2^{\val(f-g)}},
\end{equation}
where the valuation $\val(h)$
of $h$ is the smallest degree of a monomial appearing with a non-zero
coefficient in $h$. By a general theorem on Zariski rings~\cite{Zariski75}, all
the ideals of $\KXX$ are indeed closed for the topology $\topFPS$. In this
paper, we propose a constructive proof of this result, which is based on
rewriting theory. Indeed, in order to show that any ideal $I\subseteq\KXX$ is
equal to its topological closure $\olI$, we fix a finite standard basis $G$ of
$I$, \ie, a generating set of $I$ that induces a topologically confluent
rewriting system over $\KXX$. Such standard bases have the property that their
leading monomials, for a given order on monomials, divide leading monomials of
all series in $I$. In Proposition~\ref{prop:LM-I-equals-LM-I-bar}, we show how
the definition of the metric $\delta$ implies that the set of leading monomials
of $I$ and $\olI$ are equal when the order on monomials is compatible with the 
degree, in the sense that it is increasing with respect to the degree. From
this, given a commutative formal power series $f$ in $\olI$, we get a procedure
for eliminating monomials in $f$ using $G$. This procedure constructs at each
step a cofactor representation with respect to $G$ that converges to $f$. More
explicitly, for each $s_i\in G$, this yields a coefficient
$f_i^{(\infty)}\in\KXX$ proving that $f$ is indeed in $I$, as stated in our
first contribution:
\newpage

\begin{quote}
  \textbf{Theorem~\ref{thm:any-ideal-is-closed}.} \emph{Let $I\subseteq\KXX$ be
  a formal power series ideal, $G=\{s_1,\cdots,s_\ell\}$ be a finite standard
  basis of $I$ with respect to a monomial order which is compatible with the
  degree, and $f$ be an element in the topological closure of $I$ for the
  $(\X)$-adic topology. Then, the limit
  coefficients~$\left(f_1^{(\infty)},\cdots,f_\ell^{(\infty)}\right)$ of $f$ 
  relative to $G$ verify: 
  \[f=f_1^ {(\infty)}s_1+\cdots+f_\ell^{(\infty)}s_\ell.\]
  In particular, $
  I$ is closed for the $(\X)$-adic topology.
}
\end{quote}

\paragraph{Infinitary and topological confluence for commutative formal power
  series.} 

As stated above, infinitary confluence implies topological confluence. In order 
to show the equivalence of the two notions for rewriting systems over
commutative formal power series, we thus have to show the converse. For that, we  
consider a topological rewriting system on $\KXX$, where the rewriting relation
is induced by a subset $G\subseteq\KXX$ and an order on monomials that is 
compatible with the degree. In other words, we have a rewriting step $f\to h$,
if we can substitute in $f$ a leading monomial of an element of~$G$ and replace
it with the corresponding remainder to get $h$. We assume that this rewriting 
relation is topologically confluent, meaning that whenever a commutative formal
power series $f$ rewrites after finitely many rewriting steps into two
commutative formal power series $g$ and $h$, then the latter rewrite after
possibly infinite rewriting steps into a common limit $\ell$. Denoting by
$\transRew$ finite rewriting sequences and by $\topRew$ rewriting sequences that
have well-defined limits for the $(\X)$-adic topology $\topFPS$, pictorially, we
have:
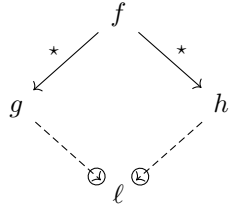
\[\begin{tikzcd}
& f\ar[dl, "\star"']\ar[dr, "\star"] &\\
g\ar[dr, dashed, circle] & & h\ar[dl, dashed, circle]\\
& \ell &
\end{tikzcd}\]
Infinitary confluence is thus represented by:
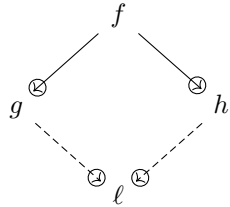
\[\begin{tikzcd}
& f\ar[dl, circle]\ar[dr, circle] &\\
g\ar[dr, dashed, circle] & & h\ar[dl, dashed, circle]\\
& \ell &
\end{tikzcd}\]
Hence, we have to prove that if we assume topological confluence only, the
dashed arrows always exist in the last diagram. The crucial observation is that
since commutative formal power series ideals are closed for $\topFPS$, the
elements $f-g$, $f-h$, and thus also $g-h$, belong to the ideal $I$ generated
by $G$. Since the assumption of topological confluence is equivalent to the fact 
that $G$ is a standard basis of $I$, leading monomials of elements of $I$ are
always divisible by leading monomials of elements of $G$, so that we can rewrite
simultaneously $g$ and $h$ as long as we obtain different results:
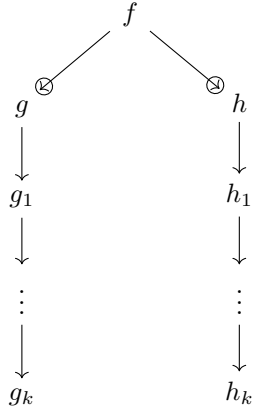
\[\begin{tikzcd}
& f\ar[dl, circle]\ar[dr, circle] &\\
g\ar[d] & & h\ar[d]\\
g_1\ar[d] &  & h_1\ar[d]\\
\vdots\ar[d] & & \vdots\ar[d]\\
g_k & & h_k
\end{tikzcd}\]
The rewriting process stops if $g_k=h_k$ at some step $k$, in such case the
dashed arrows are constructed. If not, we show that the sequences
$\left(g_k\right)_{k\in\N}$ and $\left(h_k\right)_{k\in\N}$ are in fact Cauchy
sequences for the metric $\delta$ and thus, have limits. The last part of the
proof is that these two limits are equal and our main result is stated as
follows. 
\medskip

\begin{quote}
  \textbf{Theorem~\ref{thm:top_confluence_implies_inf_confluence}.} \emph{Let
  $G\subseteq\KXX$ be a set of formal power series, $<$ be a monomial order that
  is compatible with the degree, and $\left(\KXX,\topFPS,\to\right)$ be the
  induced topological rewriting system. Then, $\to$ is $\topFPS$-confluent if 
  and only if it is infinitary confluent.
}
\end{quote}
\bigskip

\begin{center}
  \bf\large 2. CONVENTIONS AND NOTATIONS
  \addcontentsline{toc}{section}{2. Conventions and notations}
\end{center}
\setcounter{section}{2}
\bigskip

In Section 2, we recall the construction of algebras of commutative formal power
series and the notion of standard basis for an ideal of commutative formal power
series.
\smallskip

In Subsection 3.2, we will consider non-commutative formal power series.  From
now on, we will drop the adjective ``commutative'' when we consider commutative
formal power series, and simply specify the adjective ``non-commutative'' when
we consider non-commutative formal power series.  

\bigskip

\noindent
{\large\textbf{2.1. Formal power series algebras}}
\addcontentsline{toc}{subsection}{2.1. Formal power series algebras} 
\setcounter{subsection}{1}
\setcounter{theorem}{0}
\medskip

Throughout this paper, we fix a positive integer $n$, a set of indeterminates
$\{\X\}$, and a field~$\K$ of arbitrary characteristic. We denote by $[\X]$ the
free commutative monoid generated by $\{\X\}$ and by $\KX$ the algebra of
polynomials with indeterminates $\{\X\}$ and coefficients in $\K$. For any
monomial $m=x_1^{\mu_1}\cdots x_n^{\mu_n}\in[\X]$, where
${\mu_1},\cdots,{\mu_n}$ are non-negative integers, let
$\deg(m):=\mu_1+\cdots+\mu_n$ be the degree of $m$. We denote by $\coeff{f,m}$
the coefficient of the monomial $m$ in $f$ for any $f \in \KX$. Let $\supp(f)$
be the {\em support} of $f$, defined as the set of monomials that occur in $f$:  
\[\supp(f):=\left\{m\in[\X]\com\coeff{f,m}\neq 0\right\}.
\smallskip\]

We denote by $\KXX$ the algebra of formal powers series with indeterminates
$\{\X\}$ and coefficients in $\K$, defined as the Cauchy completion of $\KX$ for
the metric $\delta$ defined by  
  \begin{equation*} 
\forall f,\,g \in \KX, \quad \delta(f,g):=\frac{1}{2^{\val(f-g)}},
\end{equation*}
where, for $h \in \KX$, $\val(h)$ denotes the lowest degree of the monomials
that are in $\supp(h)$.  The coefficient of $m\in[\X]$ in a formal power 
series $f\in\KXX$ is still written $\coeff{f,m}$.
\bigskip

\noindent
{\large\textbf{2.2. Standard bases of formal power series ideals}}
\addcontentsline{toc}{subsection}{2.2. Standard bases of formal power series
  ideals}  
\setcounter{subsection}{2}
\setcounter{theorem}{0}
\medskip

Throughout this subsection, we fix a {\em monomial order} $<$ on $[\X]$, that is
a well-order such that, for all monomials $m_1,m_2,m_3 \in [\X]$, we have the
implication
\[m_1<m_2\Rightarrow m_1m_3<m_2m_3.\]
We say that $<$ is {\em compatible with the degree} if, for all monomials
$m_1,m_2$ such that $\deg(m_1)$ is strictly smaller than $\deg(m_2)$, it follows
that $m_1<m_2$. We denote by $\opord$ the opposite order of $<$. For a non-zero
$f\in\KXX$, we define the {\em leading monomial} and the {\em leading
coefficient}, that we will denote $\lm(f)$ and $\lc(f)$, as the greatest monomial for
$\opord$ that occurs in $f$ and the coefficient of $\lm(f)$ in $f$, respectively,
\ie, we have:  \[\lm(f):=\max_{\opord}\supp(f), \qquad \lc(f):=\coeff{f,\lm(f)}.
\smallskip\]
We define the {\em remainder} of $f$ as:
\[\rem(f):=\lc(f)\lm(f)-f.
\smallskip\]
Notice that either $\rem(f)=0$ or $\lm(\rem(f))\opord\lm(f)$. Moreover, we
verify the following properties: 
\begin{itemize}
\item $\forall f\in\KXX \setminus\{0\}\com \forall m\in[\X]\com\lm(m
  f)=m\lm(f)$,  
\item $\forall f, g\in\KXX\setminus\{0\}\com\lm(f+g)\opord
  \max_{\opord}\{\lm(f),\lm(g)\}$,
\item $\forall f\in\KXX\setminus\{0\}\com\forall
  \lambda\in\K\setminus\{0\}\com\lm(\lambda f)=\lm(f)$. 
\end{itemize}
\medskip

\begin{definition}
  Let $I\subseteq\KXX$ be a formal power series ideal and let $<$ be a monomial
  order on $[\X]$. A {\em standard basis} of $I$ with respect to $<$ is a subset
  $G\subseteq I$ such that for every non-zero formal power series $f\in I$,
  there exists $g\in G$ such that $\lm(g)$ divides  $\lm(f)$.  
\end{definition}

Recall from~\cite[Section 6.4]{Greuel02} that, for every ideal
$I\subseteq\KXX$, there always exists a finite standard basis of $I$.
\bigskip

\begin{center}
  \bf\large 3. TOPOLOGICAL PROPERTIES OF FORMAL POWER SERIES IDEALS
  \addcontentsline{toc}{section}{3. Topological properties of formal power
    series ideals} 
\end{center}
\setcounter{section}{3}
\bigskip

In Subsection 3.1, we provide a constructive proof of the fact that formal power
series ideals are closed for the $(\X)$-adic topology, induced by the metric
$\delta$ defined in~(\ref{equ:delta}). This is a particular case of a more
general result about Zariski rings~\cite{Zariski75}. In Subsection 3.2, we
recall from~\cite{Gerritzen98} that, for non-commutative formal power series,
there exist ideals that are not closed.  
\bigskip

\noindent
{\large\textbf{3.1. Closure in the commutative case}}
\addcontentsline{toc}{subsection}{3.1. Closure in the commutative case} 
\setcounter{subsection}{1}
\setcounter{theorem}{0}
\medskip

Let us consider an ideal $I\subseteq\KXX$ and the $(\X)$-adic topology
$\topFPS$, induced by the metric $\delta$. We shall prove that $I$ is closed for
$\topFPS$. To do that, we will work with an arbitrary standard basis of $I$, so
that we fix a monomial order $<$ on $[\X]$ and a finite standard basis
$G=\{s_1,\cdots,s_\ell\}$ of $I$ with respect to $<$.  
\smallskip

Given a subset $S\subseteq\KXX$, we denote by $\lm(S)$ the set of leading
monomials of non-zero elements of $S$:
\[\lm(S):=\{\lm(f)\com f\in S\setminus\{0\}\}.
\smallskip\]
We start by showing the following result, in which $\olI$ denotes the
topological closure of $I$ for $\topFPS$. 
\medskip

\begin{proposition}\label{prop:LM-I-equals-LM-I-bar}
  If the monomial order $<$ is compatible with the degree, then
  $\lm(I)=\lm(\olI)$.  
\end{proposition}

\begin{proof}
  Since $I \subseteq \olI$, we clearly have $\lm(I)\subseteq\lm(\olI)$. Let
  us show the converse inclusion. Let $f\in\olI$ be a non-zero formal power
  series. Then there exist formal power series in $I$ arbitrarily close to $f$
  for the metric $\delta$, in particular, there exists $g \in I$ such that
  \[\delta(g, f)<\frac{1}{2^{\deg(\lm(f))}},
  \smallskip\]
  meaning that $\deg(\lm(g-f))>\deg(\lm(f))$. Since the order $<$ is assumed to
  be compatible with the degree, we get $\lm(g-f)>\lm(f)$. Thus, for any
  monomial $m\leq\lm(f)$, we have $\coeff{g-f,m}=0$, hence, we have
  $\coeff{g,m}=\coeff{f,m}$. By definition of $\lm(f)$, it follows that, on one
  hand, $g$ is non-zero and, on the other hand, $\lm(g)=\lm(f)$. Hence,
  $\lm(f)\in\lm(I)$, which ends the proof.
\end{proof}
\medskip

\begin{remark}
  In Theorem~\ref{thm:any-ideal-is-closed}, we will show that $I$ is closed,
  from which we get that $\lm(I)=\lm(\olI)$ is true for any monomial
  order. However, we were not able to provide a constructive proof of this fact
  that works for monomial orders that are not assumed to be compatible with the
  degree. 
\end{remark}

In order to show that $I$ is closed in $\KXX$, we have to show that $\olI$ is
included in $I$, \ie, every $f\in\olI$ belongs to $I$. For that, we are going to
use Proposition~\ref{prop:LM-I-equals-LM-I-bar} to construct, for any formal
power series $f$ in $\olI$, a tuple $\left(f_i\right)_{1\leq i\leq\ell}$ of
formal power series such that $f=f_1s_1+\cdots+f_\ell s_\ell$. This will suffice
to prove that $f$ belongs to $I$ since the $s_i$ are elements of $I$. 
\smallskip

Fix $f\in \olI$ and assume from now on that the monomial order is compatible
with the degree. Note that the latter assumption is not restrictive since we can 
choose the monomial order we want and work with it to achieve our goal. 
Let us construct inductively: 
\begin{itemize}
\item a sequence $\N\ni k\mapsto\left(f_i^{(k)}\right)_{1\leq i\leq\ell}$ of
  tuples of formal power series,  
\item a sequence $\left(F_k\right)_{k\in\N}$ of formal power series in $\olI$
  and $\left(m_k\right)_{k\in\N}$ the corresponding sequence of leading
  monomials, \ie, for all $k\in\N$, $m_k:=\lm(F_k)$, 
\item a sequence $\left(i_k\right)_{k\in\N}$ of indices in $\integers{1}{\ell}$
  and a sequence $\left(q_k\right)_{k\in\N}$ of monomials. 
\end{itemize}

\begin{description}
\item[Base case:] we let, for any $i\in\integers{1}{\ell}$,
  $f_i^{(0)}:=0$. Notice how we obtain $f-\sum_{i = 1}^{\ell}  f_i^{(0)} s_i\in 
  \olI$.
\item[Induction step:] assume that, for each $i\in\integers{1}{\ell}$, the
    sequence $\left(f_i^{(k)}\right)_{k\in\N}$ is defined up to and including
    $k\in\N$ in such a way that:
  \[F_k:=f-\sum_{i = 1}^{\ell}f_i^{(k)}s_i \in \olI.
  \]
  \begin{itemize}
  \item If $F_k = 0$, we then have $f=f_1s_1+\cdots+f_\ell s_\ell$, where
    $s_i\in I$ and $f_i=f_i^{(k)}\in\KXX$. Hence, $f\in I$ and it is over.   
  \item Otherwise, the leading monomial $m_k:=\lm(F_k)$ is well-defined. Since 
    $F_k\in\olI$ and $<$ is compatible with the degree, we have $m_k\in\lm(I)$
    by Proposition~\ref{prop:LM-I-equals-LM-I-bar}. Then, as $G$ is a standard
    basis of $I$ with respect to $<$, there exist $i_k\in\integers{1}{\ell}$
    and $q_k \in[\X]$ such that:  
    \begin{equation} \label{equ:q_k}
      m_k=\lm(s_{i_k}) q_k.
    \end{equation}
  \end{itemize}
  We then define the $(k+1)$'th tuple in the sequence as:
  \[\forall i\in\integers{1}{\ell}\setminus\{i_k\}\com f_i^{(k+1)}:=f_i^{(k)},
  \smallskip\]
  and:
  \[
  f_{i_k}^{(k + 1)}:=f_{i_k}^{(k)}+\frac{\lc(F_k)}{\lc(s_{i_k})} q_k.
  \]
  Notice that:
  \[\begin{split}
  F_{k+1}
  &:=f-\sum_{i = 1}^{\ell}f_i^{(k+1)}s_i \\
  &=f-f_{i_k}^{(k)}s_{i_k}-\frac{\lc(F_k)}{\lc(s_{i_k})} q_k
  s_{i_k}-\sum_{\substack{i=1 \\ i\neq i_k}}^{\ell}f_i^{(k)}s_i\\ 
  &= \left(f-\sum_{i=1}^{\ell}f_i^{(k)}s_i\right)-\frac{\lc(F_k)}{
    \lc(s_{i_k})}q_k s_{i_k}\\ 
  &= F_k-\frac{\lc(F_k)}{\lc(s_{i_k})}q_k s_{i_k}.
  \end{split}
  \smallskip\]
  By induction hypothesis, we have $F_k\in\olI$ and since $s_{i_k}\in
  I\subseteq\olI$, we get $F_{k+1} \in \olI$. 
\end{description}

If at any step $k$, we get $F_k=0$, we obtain $f\in I$ as explained above. Thus,
assume from now on that, for all $k\in\N$, we have $F_k\neq 0$.  
\medskip

\begin{lemma}\label{lem:m_k-decreasing}
  The sequence of monomials $\left(m_k\right)_{k\in\N}$ is strictly decreasing
  for the opposite order $\opord$. 
\end{lemma}
\begin{proof}
  By construction of $i_k$ and $q_k$, we find:
  \[q_k s_{i_k}=\lc(s_{i_k})\left(\lm(s_{i_k}) q_k\right)-
  q_k \rem(s_{i_k})=\lc(s_{i_k})m_k-q_k \rem(s_{i_k}).
  \smallskip\]
  Hence, since $F_k=\lc(F_k)m_k-\rem(F_k)$:
  \[\begin{split}
  F_{k+1}
  &=F_k-\frac{\lc(F_k)}{\lc(s_{i_k})}q_k s_{i_k}\\
  &=\lc(F_k)m_k-\rem(F_k)-\lc(F_k)m_k+\frac{\lc(F_k)}{\lc(s_{i_k})}q_k
  \rem(s_{i_k})\\ 
  &=\frac{\lc(F_k)}{\lc(s_{i_k})}q_k \rem(s_{i_k})-\rem(F_k).
  \end{split}
  \smallskip\]
  And thus, from the properties of leading monomials given earlier, we get
  \[m_{k+1}=\lm(F_{k+1})=\lm\left(\frac{\lc(F_k)}{\lc(s_{i_k})}q_k
  \rem(s_{i_k})-\rem(F_k)\right)\leq_{\mathrm{op}}
  \max_{\opord}\left\{q_k\lm(\rem(s_{i_k})),\lm(\rem(F_k))\right\},
  \]
  unless, either $\rem(s_{i_k})$ or $\rem(F_k)$ is zero, in which case one of
  the leading monomial is ill-defined. In that case, we have $m_{k+1}
  \leq_{\mathrm{op}} r$ where $r$ is either $q_k\lm(\rem(s_{i_k}))$ or
  $\lm(\rem(F_k))$ according to which of the remainders is well-defined; note
  how both remainders cannot be simultaneously zero since $F_{k+1}\neq 0$. It
  then follows from the properties of remainders that: 
  \begin{itemize}
  \item $q_k\lm(\rem(s_{i_k}))\opord q_k\lm(s_{i_k})=m_k$ if $\rem(s_{i_k})$
    is non-zero and, 
  \item $\lm(\rem(F_k))\opord\lm(F_k)=m_k$ if $\rem(F_k)$ is non-zero.
  \end{itemize}
  Hence, in the end, we get: $m_{k + 1}\opord m_k$.
\end{proof}
\smallskip

Let us now define a family of $\ell$ sequences
$\left(q_1^{(k)}\right)_{k\in\N}$, $\cdots$,
$\left(q_\ell^{(k)}\right)_{k\in\N}$ containing monomials or zeros as follows: 
for any $i\in\integers{1}{\ell}$ define $q_i^{(0)} = 0$ and then for any
$k\in\N$, fix:  
\[q_i^{(k+1)}:=\frac{\lc(s_{i_k})}{\lc(F_k)}\left(f_i^{(k+1)}-f_i^{(k)}\right).
\smallskip\]
Hence, for any $k\in\N$ and $i\in\integers{1}{\ell}$, two options arise:
\begin{itemize}
\item either $i\neq i_k$, in which case $f_i^{(k+1)}=f_i^{(k)}$ and thus
  $q_i^{(k+1)}=0$, 
\item or $i=i_k$ and then $f_i^{(k+1)}=f_i^{(k)}+\frac{\lc(F_k)}{\lc(s_{i_k})}
  q_k$, and so $q_i^{(k+1)}=q_k$.  
\end{itemize}
In other words, for any $i\in\integers{1}{\ell}$ and any $k\in\N$, the monomial
$q_i^{(k+1)}$ is non-zero if and only if the monomial $m_k$ has been
``eliminated'' at Step~\eqref{equ:q_k} by choosing $\lm(s_i)$. These sequences
then verify, for any $i\in\integers{1}{\ell}$ and any $k\in\N$: 
\[f_i^{(k)}=\sum_{\substack{j = 1 \\ i_{j-1} = i}}^{k}\coeff{f_i^{(k)},q_i^{(j)}} q_i^{(j)}.
\smallskip\]

This exhibits, with $i$ and $k$ fixed, that the subsequence $Q_i^{(k)}$ of
non-zero elements from $\left(q_i^{(j)}\right)_{1\leq j\leq k}$ is exactly the
support of the formal power series $f_i^{(k)}$. Moreover, this (finite) sequence
$Q_i^{(k)}$ is strictly decreasing for the opposite order $\opord$. Indeed,
either the sequence $Q_i^{(k)}$ is empty, either it contains a single monomial,
in both cases it is then obvious, or
$Q_i^{(k)}=\left(q_i^{(j_1)},\cdots,q_i^{(j_r)}\right)$ contains $r\geq 2$
monomials where $j_1<\cdots<j_r$ are indices in $\integers{1}{k}$. Note how, for
any $s\in\integers{1}{r - 1}$, $q_i^{(j_s)}>_{\mathrm{op}}q_i^{(j_{s+1})}$.
Indeed, by definition, we have $q_i^{(j_s)} = q_{j_s - 1}$ and $q_i^{(j_{s+1})}
= q_{j_{s+1} - 1}$. But, by construction of $Q_i^{(k)}$, the
indices chosen at the step~\eqref{equ:q_k} for $k_s:=j_s-1$ and
$k'_s:=j_{s+1}-1$ verify $i_{k_s}=i_{k'_s}=i$, and so we have the equalities
$m_{k_s}=\lm(s_i) q_{k_s}$ and $m_{k'_s} = \lm(s_i) q_{k'_s}$. But
$k_s<k'_s$ and the sequence $\left(m_k\right)_{k\in\N}$ is strictly decreasing
for $\opord$ from Lemma~\ref{lem:m_k-decreasing}, and so $\lm(s_i)
q_{k_s}>_{\mathrm{op}}\lm(s_i) q_{k'_s}$, which implies
$q_{k_s}>_{\mathrm{op}}q_{k'_s}$. 
\smallskip

Therefore, by what we just proved, for $i\in\integers{1}{\ell}$ fixed, the
sequence $\left(Q_i^{(k)}\right)_{k\in\N}$ consists of finite sequences of
monomials that are strictly decreasing for the opposite order $\opord$ such
that, moreover, for any $k_1 < k_2$, $Q_i^{(k_1)}$ is an initial segment of
$Q_i^{(k_2)}$. Hence, we can consider the (possibly infinite) sequence
$Q_i^{(\infty)}$ defined as follows: if the sequence
$\left(Q_i^{(k)}\right)_{k\in\N}$ is ultimately stationary, then define
$Q_i^{(\infty)}$ as that (necessarily finite) stationary sequence value the
sequence takes; otherwise, for any arbitrary large non-negative integer $L$,
there exists $k_L \in \N$ such that, for all $k \geq k_L$, the finite sequence
$Q_i^{(k)}$ contains at least $L$ monomials and, in that case, define the
(necessarily infinite) sequence $Q_i^{(\infty)}$ as the map that associates to
any $L \in \N$ the monomial at rank $L$ in any of the sequences $Q_i^{(k)}$ for
$k \geq k_L$: this is indeed independent of the choice of the sequence
$Q_i^{(k)}$ by the previous initial segments discussion and, thus, yields a
well-defined infinite sequence $Q_i^{(\infty)}$. Finally, by construction, the
sequence $Q_i^{(\infty)}$ of monomials, regardless of whether it is finite or
not, is strictly decreasing for the opposite order $\opord$.
\smallskip

These facts on these new sequences entail the following proposition. 
\medskip

\begin{proposition}\label{prop:cauchy-sequences}
  For any $i\in\integers{1}{\ell}$ fixed, the sequence
  $\left(f_i^{(k)}\right)_{k\in\N}$ is a Cauchy sequence. 
\end{proposition}

\begin{proof}
  Using the definition of the sequence $\left(q_i^{(k)}\right)_{k\in\N}$, we
  obtain 
  \[f_i^{(k+1)}-f_i^{(k)}=\frac{\lc(F_k)}{\lc(s_{i_k})}q_i^{(k+1)},
  \]
  in such a way that for any positive integers $k_1<k_2$, we get:
  \begin{equation}\label{equ:f_i:k_2-k_1}
    \begin{split}
      f_i^{(k_2)}-f_i^{(k_1)}
      &=\sum_{j = k_1}^{k_2-1}\left(f_i^{(j+1)}-f_i^{(j)}\right)
      =\sum_{j = k_1}^{k_2-1}\frac{\lc(F_j)}{\lc(s_{i_j})}q_i^{(j+1)}.
    \end{split}
  \end{equation}
  Thus, either the sequence $\left(f_i^{(k)}\right)_{k\in\N}$ is ultimately
  stationary, it is the case if and only if the index $i$ has been chosen only
  finitely many times among the infinite times we went through the
  step~\eqref{equ:q_k}, and in which case the sequence is obviously a Cauchy
  sequence. Or the sequence is not ultimately stationary and then, for any $k_1
  \in \N$, there will always exist $k_2>k_1$ such that
  $\delta\left(f_i^{(k_2)},f_i^{(k_1)}\right)>0$. However, for any real number
  $\varepsilon>0$, we can fix
  \[K_\varepsilon:=\min\left\{k\in\N\setminus\{0\}\com q_i^{(k)}\neq 0\
  \text{and}\ \deg\left(q_i^{(k)}\right)>
  \log_2\left(\frac1\varepsilon\right)\right\}.
  \smallskip\]
  The integer $K_\varepsilon$ is well-defined since
  $\left(f_i^{(k)}\right)_{k\in\N}$ is assumed to be not ultimately stationary
  and the sequence $Q_i^{(\infty)}$, which contains all non-zero $q_i^{(k)}$, is
  strictly decreasing for $\opord$, so that the degrees of the monomials from
  $Q_i^{(\infty)}$ are unbounded because there are finitely many variables and
  the order is compatible with the degree. It follows that, for any $k_2>k_1\geq
  K_\varepsilon$, either $f_i^{(k_2)}=f_i^{(k_1)}$ in which case we have 
  $\delta\left(f_i^{(k_2)},f_i^{(k_1)}\right)=0<\varepsilon$, or, the 
  monomial $\lm\left(f_i^{(k_2)}-f_i^{(k_1)}\right)$ is well-defined and
  satisfies, by Formula~\eqref{equ:f_i:k_2-k_1}: 
  \[\lm\left(f_i^{(k_2)}-f_i^{(k_1)}\right)\leq\max_{\opord}
  \left\{q_i^{(j+1)}\com q_i^{(j+1)}\neq 0\ \text{and}\ k_1\leq j<k_2\right\}.
  \smallskip\]
  This maximum is well-defined because the set of the right hand side of the
  inequality is finite and non-empty since $f_i^{(k_2)}\neq
  f_i^{(k_1)}$. Denoting by $q_i^{(j_0+1)}$ that maximum, we get $j_0+1>k_1\geq
  K_\varepsilon$. Since the sequence $Q_i^{(\infty)}$ is strictly decreasing for
  $\opord$, we have $q_i^{(j_0+1)}\opord q_i^{(K_\varepsilon)}$, and thus
  \[\deg\left(q_i^{(j_0+1)}\right)\geq\deg\left(q_i^{(K_\varepsilon)}\right)>
  \log_2\left(\frac1\varepsilon\right),
  \smallskip\]
  because the order is compatible with the degree. Since
  $q_i^{(j_0+1)}=\lm\left(f_i^{(k_2)}-f_i^{(k_1)}\right)$, the formula means 
  \[\delta\left(f_i^{(k_2)},f_i^{(k_1)}\right)<\varepsilon.
  \smallskip\]
  Since this is true for every $k_2>k_1$ and every $\varepsilon>0$, the sequence
  $\left(f_i^{(k)}\right)_{k\in\N}$ is a Cauchy sequence.
\end{proof}
\smallskip

From Proposition~\ref{prop:cauchy-sequences}, since $\KXX$ is Cauchy-complete,
for any $i\in\integers{1}{\ell}$, the sequence~$\left(f_i^{(k)}\right)_{k\in\N}$
converges to a limit we denote $f_i^{(\infty)}$. 
\begin{definition}
  The elements $\left(f_1^{(\infty)},\cdots,f_\ell^{(\infty)}\right)$ are called
  the {\em limit coefficients of $f$ relative to $G$}.
\end{definition}

We are now able to prove the main result of the section.
\smallskip

\begin{theorem}\label{thm:any-ideal-is-closed}
  Let $I\subseteq\KXX$ be a formal power series ideal, $G=\{s_1,\cdots,s_\ell\}$
  be a finite standard basis of $I$ with respect to a monomial order which is
  compatible with the degree, and $f$ be an element in the topological closure
  of $I$ for the $(\X)$-adic topology. Then, the limit
  coefficients~$\left(f_1^{(\infty)},\cdots,f_\ell^{(\infty)}\right)$ of $f$
  relative to $G$ verify: 
  \[f=f_1^ {(\infty)}s_1+\cdots+f_\ell^{(\infty)}s_\ell.\]
  In particular, $
  I$ is closed for the $(\X)$-adic topology.
\end{theorem}

\begin{proof}
  By continuity of algebraic operations, the sequence
  $\left(f-\sum_{i=1}^\ell f_i^{(k)}s_i\right)_{k\in\N}$ converges to
  \[\lim_{k\to\infty}\left(f-\sum_{i=1}^\ell f_i^{(k)}s_i\right)=
  f-\sum_{i=1}^\ell f_i^{(\infty)}s_i.
  \smallskip\]
  But then, on one hand:
  \[\lim_{k\to\infty}\delta\left(f-\sum_{i = 1}^\ell f_i^{(k)}s_i,0\right)=
  \frac{1}{2^{\lim_{k\to\infty}\deg\left(\lm\left(f-\sum_{i=1}^\ell
      f_i^{(k)}s_i\right)\right)}}=\frac{1}{2^{\lim_{k\to\infty}\deg(m_k)}}.
  \smallskip\]
  On the other hand, Lemma~\ref{lem:m_k-decreasing} shows that the sequence
  $\left(m_k\right)_{k\in\N}$ is strictly decreasing for $\opord$. Thus, since
  the order is compatible with the degree and since we have finitely many
  indeterminates, the sequence
  $\left(\deg(m_k)\right)_{k\in\N}$ tends to infinity, and thus, we have: 
  \[0= \lim_{k\to\infty}\delta\left(f-\sum_{i=1}^\ell f_i^{(k)}s_i,0\right)
  =\delta\left(\lim_{k\to\infty}\left(f-\sum_{i=1}^\ell f_i^{(k)}s_i\right),0 
  \right) = \delta\left( f-\sum_{i=1}^\ell f_i^{(\infty)}s_i,0
  \right).
  \smallskip\]
We finally get $f=f_1^{(\infty)}s_1+\cdots+ f_\ell^{(\infty)} s_\ell$, which
  proves that $f$ belongs to $I$. 
\end{proof}

\bigskip

\noindent
{\large\textbf{3.2. A counter-example in the non-commutative case}}
\addcontentsline{toc}{subsection}{3.2. A counter-example in the non-commutative
  case}  
\setcounter{subsection}{2}
\setcounter{theorem}{0}
\medskip

In this subsection, we recall from~\cite{Gerritzen98} that ideals of
non-commutative formal power series algebras are not necessarily
closed. We also explain why the proof given in Subsection 3.1 does not
translate in the non-commutative case.
\smallskip

Non-commutative formal power series are constructed in the same way as 
commutative ones, where we replace the polynomial algebra $\KX$ by the tensor
algebra $\KXnc$ over the vector space with basis $\{\X\}$. As in the commutative 
case, the distance $\delta$ on $\KXnc$ is defined by formula
$\delta(f,g):=2^{-\val(f-g)}$, and $\KXXnc$ is the Cauchy-completion of $\KXnc$
for $\delta$.
\smallskip

Let $\Kxy$ be the algebra of non-commutative formal power series in two
variables $x$ and $y$ and consider the two-sided ideal $I$ generated by
$y$. Then, from the proof of~\cite[Lemma 1.2]{Gerritzen98}, the series
\[\sum_{n\in\N}x^nyx^n,
\smallskip\]
does not belong to $I$, but it belongs to the topological closure $\olI$ of $I$
for the $(x,y)$-adic topology induced by $\delta$. That shows that $I$ is
not closed in $\Kxy$ for the $(x,y)$-adic topology.
\smallskip

Notice how this situation would not arise in the commutative case. Indeed, we
have for commutative monomials: 
\[\sum_{n\in\N}x^nyx^n=\sum_{n\in\N}x^{2n}y=\left(\sum_{n\in\N}x^{2n}\right)y\in I.
\smallskip\]

The first point where the procedure described in the previous section fails to
translate in the non-commutative case is that a non-commutative formal power
series ideal does not necessarily admit a finite standard basis. However, even
if such a finite standard basis exists, our procedure still does not translate 
in the non-commutative setting: indeed, Step~\eqref{equ:q_k} would require two
monomials $q_k^{\mathrm{left}}$ and $q_k^{\mathrm{right}}$ to factorise the
leading monomial $m_k$ instead of a single one; furthermore, bookkeeping of the
factorisations of reduced monomials is made much harder in the sense that, in
contrast with the commutative case where a single formal power series
$f_i^{(k)}$ is enough for each $i$ and each $k$ to keep track of all the
previous factorisations of reduced monomials up to step $k$, in the worst case,
as examplified by the counter-example above, we have to keep track of each
factorisation independently of each other, giving rise to two $k$-tuples of
formal power series $\left(L_i^{(1)}, {\cdots}, L_i^{(k)}\right)$ and
$\left(R_i^{(1)}, {\cdots}, R_i^{(k)}\right)$, for each $i \in
\integers{1}{\ell}$, satisfying:  
\[F_k=f-\sum_{i = 1}^{\ell}\sum_{j=1}^kL_i^{(j)} s_i R_i^{(j)}.
\]
With that in mind, in the case where none of the $F_k$'s are zero, considering
the limit process, we would have
\[f=\sum_{i=1}^{\ell}\sum_{j=1}^{\infty}L_i^{(j)} s_i R_i^{(j)}.
\]
However, this is not a cofactor representation of $f$ with respect to the
$s_i$'s since it is not a finite combination of the $s_i$'s with formal power
series coefficients.
\bigskip

\begin{center}
  \bf\large 4. APPLICATION TO TOPOLOGICAL REWRITING THEORY
  \addcontentsline{toc}{section}{4. Application to topological rewriting theory}
\end{center}
\setcounter{section}{4}
\bigskip

In Section 4, we show that topological confluence and infinitary confluence 
are equivalent notions for rewriting over formal power series. These two
properties provide characterisations of standard bases.  
\bigskip

\noindent
{\large\textbf{4.1. Topological and infinitary confluence}}
\addcontentsline{toc}{subsection}{4.1. Topological and infinitary confluence} 
\setcounter{subsection}{1}
\setcounter{theorem}{0}
\medskip

In this subsection, we recall the definition of topological rewriting systems as 
well as various notions of confluence associated with them. We also show how
these notions are related to each other.
\smallskip

We recall that a {\em topological rewriting system} is a triple $(A,\tau,\to)$,
where $(A,\tau)$ is a topological space and $\to\ \subseteq A\times A$ is a
binary relation on $A$, called {\em rewriting relation}. We denote by
$\transRew$ the reflexive transitive closure of $\to$ and $\topRew$ the
topological closure of $\transRew$ for the product topology $\dTop{A}\times\tau$
on $A\times A$, where $\dTop{A}$ is the discrete topology on $A$. In other
words, we have $a\transRew b$ if and only if there exist an integer $k\in\N$ and
elements $a_0,\cdots,a_k\in A$, such that $a=a_0\to a_1\to\cdots\to a_k=b$. The
integer $k$ is called the {\em length} of the sequence and the case $k=0$ means
that $a=b$. Moreover, we have $a\topRew b$ if and only if every neighbourhood
$V$ of $b$ for $\tau$ contains $b'\in V$ such that $a\transRew b'$. An element
$a\in A$ is called a {\em normal form} for $\to$ if there is no $b\in A$ such
that $a \to b$.
\medskip

\begin{definition}
  Let $(A,\tau,\to)$ be a topological rewriting system.
  \begin{enumerate}
  \item The rewriting relation $\to$ is said to be {\em confluent} if, for every
    $a,b,c\in A$ such that we have $a\transRew b$ and $a\transRew c$,
    there exists $d\in A$ such that $b\transRew d$ and $c\transRew d$:
    \[\begin{tikzcd}
    & a\ar[dl, "\star"']\ar[dr, "\star"] &\\
    b\ar[dr, dashed, "\star"'] & & c\ar[dl, dashed, "\star"]\\
    & d &
    \end{tikzcd}\]
  \item The rewriting relation $\to$ is said to be {\em $\tau$-confluent} if, for
    every $a,b,c\in A$ such that we have $a\transRew b$ and $a\transRew c$,
    there exists $d\in A$ such that $b\topRew d$ and $c\topRew d$:
    \[\begin{tikzcd}
    & a\ar[dl, "\star"']\ar[dr, "\star"] &\\
    b\ar[dr, dashed, circle] & & c\ar[dl, dashed, circle]\\
    & d &
    \end{tikzcd}\]
  \item The rewriting relation $\to$ is said to be {\em infinitary confluent}
      if, for every $a,b,c\in A$ such that we have $a\topRew b$ and $a\topRew
      c$, there exists $d\in A$ such that $b\topRew d$ and $c\topRew d$:
    \[\begin{tikzcd}
    & a\ar[dl, circle]\ar[dr, circle] &\\
    b\ar[dr, dashed, circle] & & c\ar[dl, dashed, circle]\\
    & d &
    \end{tikzcd}\]
  \end{enumerate}
\end{definition}

\begin{remark} \label{rmk:NF}
    Let us highlight one consequence of infinitary confluence that will
    subsequently be used to provide counter-examples. Assume that the
    topological space $(A, \tau)$ is a $T_1$-space, that is to say, for any $a
    \in A$, the intersection of all neighbourhoods of $a$ is exactly $\{a\}$.
    Then, infinitary confluence of the topological rewriting system $(A, \tau,
    \to)$ implies that, for any $a,b,c\in A$, where $b$ and $c$ are normal forms
    for $\to$ such that $a \topRew b$ and $a \topRew c$, we have necessarily $b
    = c$. Indeed, by infinitary confluence, we know there exists $d \in A$ such
    that $b \topRew d$ and $c \topRew d$. That is to say, for any neighbourhood
    $U$ of $d$ in $(A, \tau)$, there exist $b', c' \in U$ such that $b \transRew
    b'$ and $c \transRew c'$. However, since $b$ and $c$ are, by hypothesis,
    normal forms for $\to$, then $b = b'$ and $c = c'$. In other words, we have
    that $b$ and $c$ belong to any neighbourhood of $d$ in $(A, \tau)$. Finally,
    by assumption, $(A, \tau)$ is a $T_1$-space, and, hence, $b = d = c$.
\end{remark}

Since $\transRew\ \subseteq\topRew$, confluence and infinitary confluence both
imply $\tau$-confluence. The converse implications are both false in general. A
topologically confluent rewriting system that is not confluent is given by
Example~\ref{equ:inf_conf_does_not_imply_conf} which requires to recall how
topological rewriting systems over formal power series are defined. Examples of
topologically confluent rewriting systems that are not infinitary confluent are
given in Example~\ref{ex:cyclic},~\ref{ex:Nbar}
and~\ref{equ:inf_conf_does_not_imply_conf}. 

\begin{example} \label{ex:cyclic}
  Consider the following counter-example, given by $X=[0,2]\subseteq\R$ with the
  usual topology $\tau$ and $\to$ is defined by
  \[\begin{tikzcd}
  \frac{1}{2^{n+1}}\ar[r, bend left] & \frac{1}{2^n}\ar[l, bend left] &
  \text{and} &
  2-\frac{1}{2^n}\ar[r, bend left] & 2-\frac{1}{2^{n+1}},\ar[l, bend left]
  \end{tikzcd}
  \smallskip\]
  for every $n\in\N$. Hence, we have:
  \[\begin{tikzcd}
  0 &
  \cdots\ar[r, bend left]\ar[l, circle] &
  \frac18\ar[r, bend left]\ar[l, bend left] &
  \frac14\ar[l, bend left]\ar[r, bend left] &
  \frac12\ar[l, bend left]\ar[r, bend left] &
  1\ar[l, bend left]\ar[r, bend left] &
  \frac32\ar[l, bend left]\ar[r, bend left] &
  \frac74\ar[l, bend left]\ar[r, bend left] &
  \frac{15}{8}\ar[l, bend left]\ar[r, bend left] &
  \cdots\ar[l, bend left]\ar[r, circle] & 2
  \end{tikzcd}
  \smallskip\]
  This rewriting relation is confluent, hence $\tau$-confluent, because every
  finite rewriting sequence can be reversed, but it is not infinitary confluent
  according to Remark~\ref{rmk:NF}. Indeed, the space $X$ is a $T_1$-space
  (being a subspace of the real line) and $0$ and $2$ are two distinct normal
  forms for $\to$ such that we both have $1 \topRew 0$ and $1 \topRew 2$.
  \end{example}

In the counter-example provided in Example~\ref{ex:cyclic}, $\to$ is
$\tau$-confluent in part because it is cyclic, \ie, we have rewriting loops
$a\transRew a$ of length at least $1$. Hence, one can naturally wonder if an
example of a rewriting relation with no rewriting loop of length at least $1$ 
over a $T_1$-space, that is topologically confluent but not 
infinitary confluent exists. The answer is yes, as shown by the following
example. 
\medskip

\begin{example}\label{ex:Nbar}
  Consider $X=\left(\Ninfty\right)\times\left(\Ninfty\right)$, equipped with the
  product topology $\tau$ of the order topology over $\Ninfty$. This is a
  Hausdorff space and, hence, a $T_1$-space. A basis for the order topology over
  $\Ninfty$ is given by the sets  
  \[\{n\in\N\com a<n<b\},\qquad \{n\in\N\com n<b\},\qquad
  \{n\in\Ninfty\com a<n\},
  \smallskip\]
  where $a\in\N$ and $b\in\Ninfty$. Then, consider the rewriting relation $\to$
  on $X$ given by
  \[(n,m)\to(n+1,m)\andquad (n,m)\to (n,m+1),
  \smallskip\]
  whenever $n,m\in\N$. The rewriting relation $\to$ is confluent, hence
  $\tau$-confluent, because we have a rewriting path $(n,m)\transRew (n',m')$ if
  and only if $n,n',m,m'\in\N$ are such that $n\leq n'$ and $m\leq m'$, so that
  we have confluence diagrams:
  
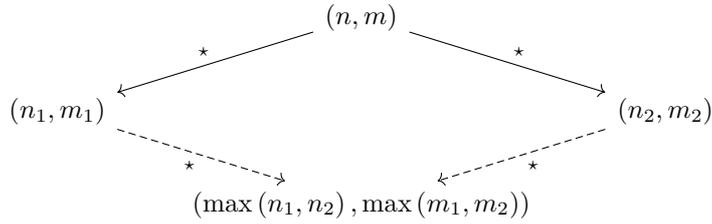
\[\begin{tikzcd}
  & (n,m)\ar[dl, "\star"']\ar[dr, "\star"] &\\
  \left(n_1,m_1\right)\ar[dr, dashed, "\star"'] & &
  \left(n_2,m_2\right)\ar[dl, dashed, "\star"]\\
  & \left(\max\left(n_1,n_2\right),\max\left(m_1,m_2\right)\right) &
  \end{tikzcd}\]
  But $\to$ is not infinitary confluent since we have
  \[\begin{tikzcd}
  (\infty,0) & \cdots\ar[l, circle] & (2,0)\ar[l] & (1,0)\ar[l] &
  (0,0)\ar[l]\ar[r] & (0,1)\ar[r] & (0,2)\ar[r] & \cdots\ar[r, circle] &
  (0,\infty),
  \end{tikzcd}
  \smallskip\]
  and $(\infty,0)$ and $(0,\infty)$ are distinct normal forms for $\to$ such
  that $(0, 0) \topRew (\infty,0)$ and $(0,0) \topRew (0,\infty)$ which,
  by Remark~\ref{rmk:NF}, would need to be equal if the system was infinitary
  confluent.
\end{example}
\bigskip

\noindent
{\large\textbf{4.2. Confluence for rewriting on formal powers series}}
\addcontentsline{toc}{subsection}{4.2. Confluence for rewriting on formal powers
  series} 
\setcounter{subsection}{2}
\setcounter{theorem}{0}
\medskip

In this subsection, we recall the construction of topological rewriting systems
over formal power series and prove that in this setting, $\tau$-confluence and 
infinitary confluence are equivalent properties. 
\smallskip

As in the beginning of the paper, we fix a finite set of indeterminates
$\{\X\}$. For a given monomial order $<$ on $[\X]$ that is compatible with the 
degree, and a fixed set $G$ of non-zero formal power series, we define the
following rewriting relation $\to$ on $\KXX$:  
\[\lambda\big(m\lm(s)\big)+S\to
\frac{\lambda}{\lc(s)}\big(m \rem(s)\big)+S,
\smallskip\]
where:
\begin{itemize}
\item $\lambda\in\K\setminus\{0\}$ is a non-zero scalar,
\item $m\in [\X]$ is a monomial,
\item $s\in G$ is a non-zero formal power series, 
\item $S\in\KXX$ is a formal power series such that $m\lm(s) \not \in \supp(S)$.   
\end{itemize}
We get a topological rewriting system $\left(\KXX,\topFPS,\to\right)$, where
$\topFPS$ is the $(\X)$-adic topology induced by the metric $\delta$ defined in
\eqref{equ:delta}. Recall from~\cite[Theorem 4.1.3]{Chenavier20} that a subset
$G$ of $\KXX$ is a standard basis of the ideal it generates with respect to the
monomial order $<$ if and only if the rewriting relation $\to$ is
$\topFPS$-confluent. In this setting, we are able to introduce a topological
rewriting system that is topologically confluent but not confluent.  
\smallskip

\begin{example}\label{equ:inf_conf_does_not_imply_conf}
  Consider the formal power series algebra $\K[[x,y,z]]$ with the deg-lex
  monomial order induced by $x>y>z$; in particular, it is compatible with the
  degree. Moreover, we let
  \[G=\left\{z-y,z-x,y-y^2,x-x^2\right\}.\]
  Then, every element of $I(G)$ has zero constant coefficient, so leading
  monomials of non-zero elements of $I(G)$ are either divisible by
  $z=\lm(z-x)=\lm(z-y)$, by $y=\lm\left(y-y^2\right)$, or by
  $x=\lm\left(x-x^2\right)$, proving that $G$ is a standard basis of
  $I(G)$ with respect to $<$. In particular, the topological rewriting system
  on $\K[[x,y,z]]$ induced by $G$ is $\topFPS$-confluent. However, it is not
  confluent because $z$ both rewrites into $y$ and $x$, and finite rewriting
  sequences starting with $y$ and $x$ finish with some powers of $y$ and $x$,
  respectively. These observations are summarised in the following diagram:
  
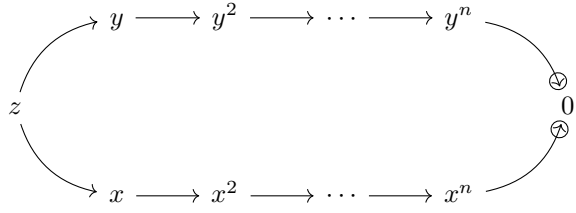
\[\begin{tikzcd}
  & y\ar[r] & y^2\ar[r] & \cdots\ar[r] & y^n\ar[rd, bend left, circle]\\
  z\ar[ru, bend left]\ar[rd, bend right] & & & & & 0\\
  & x\ar[r] & x^2\ar[r] & \cdots\ar[r] & x^n\ar[ru, bend right, circle]
  \end{tikzcd}\]
\end{example}

In Theorem~\ref{thm:top_confluence_implies_inf_confluence}, we show the main
result of Section 4, stating that infinitary confluence and topological
confluence are equivalent properties in the context of formal power
series. For that, we fix a subset $G\subseteq\KXX$ and we denote by $I$ the
formal power series ideal generated by $G$:
\[I:=I(G)\subseteq\KXX.
\smallskip\]   
We first establish the following result, which is the topological adaptation 
of a well-known result in the context of Gröbner bases theory~\cite[Theorem
  8.2.7]{Baader98}. Note that the topological closure of formal power series
ideals is used at the end of the proof. 
\medskip

\begin{proposition}\label{prop:f-h-in-I-bar}
  For all $f,g\in\KXX$, if $f\topRew g$, then $f-g\in I$.
\end{proposition}

\begin{proof}
  First, if $g=f$, then there is nothing to prove. Second, if $f\to g$, then we
  have 
  \[f=\lambda(m\lm(s))+S, \qquad g=\frac{\lambda}{\lc(s)}(m \rem(s))+S,
  \]
  for $\lambda\in\K\setminus\{0\}$, $m\in [\X]$, $s\in G$, and $S\in\KXX$
  such that $m\lm(s)\notin\supp(S)$. Cancellations ensue in the computation 
  of $f-g$, and we obtain:  
  \[f-g=\frac{\lambda}{\lc(s)}m  s.
  \]
  But $s\in G\subseteq I$ and $I$ is an ideal, therefore $f-g\in I$. Third, if
  $f\transRew g$ and $f\neq g$, then, by induction on the length $k\geq 1$ of
  the rewriting sequence $f=f_0\to f_1\to\cdots\to f_k=g$, we have $f-g\in I$.
  Finally, if we have $f\topRew g$, then, for every integer $k\in\N$, there
  exists $f_k\in\KXX$ such that 
  \[\delta(f_k,g)<\frac{1}{2^k},
  \smallskip\]
and  $f\transRew f_k$. The sequence $\left(f_k\right)_{k\in\N}$ then converges
to $g$. From the third case treated in the proof, for every $k\in\N$, we have
$f-f_k\in I$, so that $f-g=\lim_{k\to\infty}(f-f_k)$ belongs to $\olI$. Now,
from Theorem~\ref{thm:any-ideal-is-closed}, the ideal $I$ is closed, so that
$f-g\in I$, which completes the proof. 
\end{proof}
\smallskip

We are now in position to prove the main result of Section 4.
\medskip

\begin{theorem}\label{thm:top_confluence_implies_inf_confluence}
  Let $G\subseteq\KXX$ be a set of formal power series, $<$ be a monomial order
  that is compatible with the degree, and $\left(\KXX,\topFPS,\to\right)$ be the
  induced topological rewriting system. Then, $\to$ is $\topFPS$-confluent if
  and only if it is infinitary confluent.
\end{theorem}

\begin{proof}
  We only have to show that if $\to$ is $\topFPS$-confluent, then it is also
  infinitary confluent. Hence, we assume that $\to$ is $\topFPS$-confluent,
  which, from~\cite[Theorem 4.1.3]{Chenavier20}, means that $G$ is a standard
  basis of the ideal $I=I(G)$ it generates.
  \smallskip

  Let $f,g,h\in\KXX$ be formal power series such that:
  
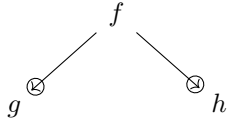
\[\begin{tikzcd}
  & f\ar[dl, circle]\ar[dr, circle] &\\
  g & & h
  \end{tikzcd}
  \smallskip\]

 In the following, we shall define two sequences $\left(g_k\right)_{k\in\N}$ and $\left(h_k\right)_{k\in\N}$ of
  elements of $\KXX$ such that $g\transRew g_k$ and $h\transRew h_k$, for every
  $k\in\N$. Then, we will show that these two sequences are Cauchy sequences, hence have
  limits, and that these two limits are equal, which will conclude the proof. We
  construct the two sequences by the following recursive procedure.  
  \medskip
  
  \begin{description}
  \item[Base step:] we let $g_0:=g$ and $h_0:=h$.
  \item[Recursive step:] let $k\in\N$ and assume that we have defined rewriting
    sequences 
    \[g=g_0\reflRew g_1\reflRew\cdots\reflRew g_k, \qquad
    h=h_0\reflRew h_1\reflRew\cdots\reflRew h_k,
    \smallskip\]
    where $\reflRew$ is the reflexive closure of $\to$, \ie, $a\reflRew b$ if
    $a=b$ or if $a\to b$. If we have $g_k-h_k=0$, we define all subsequent terms
    of the sequences as $g_l:=g_k=h_k$ and $h_l:=h_k=g_k$, for every $l\geq k$.
    Otherwise, if $g_k-h_k\neq 0$, we let
    \[m_k:=\lm\left(g_k-h_k\right).
    \smallskip\]
    From Proposition~\ref{prop:f-h-in-I-bar}, the formal
    power series $f-g$, $f-h$, $g-g_k$, and $h-h_k$ belong to $I$, so that:
    \[g_k-h_k=(g_k-g)+(g-f)+(f-h)+(h-h_k)\in I.
    \smallskip\]
    Since $G$ is a standard basis of $I$, there exist $s\in G$ and $m\in[\X]$
    such that $m_k=m\lm(s)$. Then, we define $g_{k+1}$ and $h_{k+1}$ by
    rewriting the monomial $m_k$ if possible, \ie, we let
    \[\begin{split}
    g_{k+1}&:=g_k- \coeff{g_k,m_k}\left(m_k-\frac{1}{\lc(s)}(m
    \rem(s))\right),\\
    h_{k+1}&:=h_k- \coeff{h_k,m_k}\left(m_k-\frac{1}{\lc(s)}(m
    \rem(s))\right).
    \end{split}
    \smallskip\]
    Indeed, we have either $g_{k+1}=g_k$ or $g_k\to g_{k+1}$, depending on 
    $\coeff{g_k,m_k}$ is equal to zero or not, hence we have $g_k\reflRew
    g_{k+1}$. In the same manner, we have $h_k\reflRew h_{k+1}$. Since
    $m_k\in\supp(g_k-h_k)$, we cannot have $g_{k+1}=g_k$ and
    $h_{k+1}=h_k$. Moreover, notice that $m_{k+1}\opord m_{k}$ because
    $\coeff{g_{k},m}=\coeff{h_{k},m}$ for every $m_{k}\opord m$ and
    $m_{k}$, which does not belong to $\supp(g_{k+1})\cup\supp(h_{k+1})$, rewrites
    into a series consisting of monomials that are strictly smaller than $m_{k}$
    for $\opord$, so that $\coeff{g_{k+1},m}=\coeff{h_{k+1},m}$ for every
    $m_{k}\leq_{\text{op}} m$. Hence, we have $m_{k+1}\opord m_{k}$.
    Finally,
    since $g_{k+1}$ and $h_{k+1}$ are equal or are obtained from $g_{k}$ and $h_{k}$
    by rewriting $m_{k}$ into a series consisting of monomials that are strictly
    smaller than $m_{k}$ for $\opord$, we have
    $\coeff{g_{k+1},m}=\coeff{g_{k},m}$ and $\coeff{h_{k+1},m}=\coeff{h_{k},m}$ for
    $m_{k}\opord m$.
  \end{description}
  If for some $k\in\N$ we have $g_k-h_k=0$, then letting $\ell=g_k=h_k$, we have  
  $g\transRew\ell$ and $h\transRew\ell$, and thus:
  
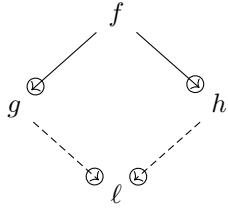
\[\begin{tikzcd}
  & f\ar[dl, circle]\ar[dr, circle] &\\
  g\ar[dr, dashed, circle] & & h\ar[dl, dashed, circle]\\
  & \ell &
  \end{tikzcd}
  \smallskip\]
  Now, assume that for every $k\in\N$, we have $g_k-h_k\neq 0$. Since for each
  $k\in\N$, we have either $g_{k+1}\neq g_k$ or $h_{k+1}\neq h_k$, then at least
  one of the two sequences is not ultimately stationary. However, both of these
  sequences are Cauchy sequences. This is obvious if a sequence is ultimately
  stationary. If not, say, for instance, $\left(g_k\right)_{k\in\N}$ is not
  ultimately stationary, we first prove by induction on $i\in\N$ that we have,
  for every $k \in \N$ and for every monomial  $m$ such that $m_k\opord m$:
  \begin{equation}\label{equ:g_m_k_i}
    \forall i\in\N:\quad
    \coeff{g_k,m}=\coeff{g_{k+i},m}. 
  \end{equation}
  For $i=0$, this is obvious. Assume~\eqref{equ:g_m_k_i} for $i\in\N$. Since
  $\coeff{g_{k+i+1},m}=\coeff{g_{k+i},m}$ for every monomial $m$ such that
  $m_{k+i}\opord m$ and since $m_{k+i}\leq_\text{op}m_k$, because
  $\left(m_k\right)_{k\in\N}$ is strictly decreasing for $\opord$, we have
  \[\coeff{g_{k+i+1},m}=\coeff{g_{k+i},m}=\coeff{g_k,m},
  \smallskip\]
  for $m_k\opord m$, and the induction is over. As $\opord$ is compatible
  with the degree, from~\eqref{equ:g_m_k_i}, we have:
  \[\forall i,k\in\N:\quad \delta\left(g_k,g_{k+i}\right)\leq
  \frac{1}{2^{\deg(m_k)}}.
  \smallskip\]
  Since $\left(m_k\right)_{k\in\N}$ is strictly decreasing, we get that this
  distance goes to $0$ as $k,i\to\infty$, \ie, $\left(g_k\right)_{k\in\N}$ is a
  Cauchy sequence. By the same reasoning, we show that
  $\left(h_k\right)_{k\in\N}$ is also a Cauchy sequence. Using the
  Cauchy-completeness of $\KXX$, the two sequences have limits, denoted by
  $g_\infty$ and $h_\infty$. Moreover, by construction, we have $g\topRew
  g_\infty$ and $h\topRew h_\infty$. 
  \smallskip

  It remains to show that $g_\infty=h_\infty$. But we have that $\delta(g_k,
  h_k) = 2^{-\deg(m_k)}$ for every $k \in \N$ and, by general metric spaces
  properties, the metric $\delta$ is continuous on $\KXX^2$ equipped with the
  product topology $\topFPS \times \topFPS$. Hence, $\delta(g_\infty, h_\infty)
  = \lim_{k \to \infty} \delta(g_k,h_k) = \lim_{k\to\infty} 2^{-\deg(m_k)}$.
  Finally, since the sequence $\left(m_k\right)_{k\in\N}$ is strictly decreasing
  for $\opord$ and since the order is compatible with degree, it follows that
  $\lim_{k\to\infty} \deg(m_k) = \infty$ and, thus, $\delta(g_\infty, h_\infty)
  = 0$ which allows us to conclude $g_\infty = h_\infty$. Denoting that common
  value by $\ell=g_\infty=h_\infty$, we thus have:
  \[\begin{tikzcd}
  & f\ar[dl, circle]\ar[dr, circle] &\\
  g\ar[dr, dashed, circle] & & h\ar[dl, dashed, circle]\\
  & \ell &
  \end{tikzcd}
  \smallskip\]
  Hence, $\to$ is infinitary confluent.
\end{proof}

From~\cite[Theorem 4.1.3]{Chenavier20}, the rewriting relation $\to$ induced by
a monomial order $<$ and a set $G$ of formal power series is $\topFPS$-confluent
if and only if $G$ is a standard basis relative to $<$ of the ideal it
generates. Hence, we get the following corollary, whose ``only-if'' part is in
fact what was shown in the proof of
Theorem~\ref{thm:top_confluence_implies_inf_confluence}.
\medskip

\begin{corollary}
  A set $G\subseteq\KXX$ is a standard basis, relative to a monomial order $<$
  that is compatible with the degree, of the ideal it generates if and only if
  the rewriting relation induced by $G$ and $<$ is infinitary confluent.
\end{corollary}
\bibliography{biblio}

\end{document}